\documentclass{article}
\usepackage{amsthm,amssymb,amsmath,enumerate,amsthm}
\usepackage[T1]{fontenc}
\usepackage[utf8]{inputenc}
\usepackage{authblk}

\newcommand{\cC}{\ensuremath{\mathcal{C}}}

\newcommand{\cM}{\ensuremath{\mathcal{M}}}
\newcommand{\cN}{\ensuremath{\mathcal{N}}}

\newcommand{\cS}{\ensuremath{\mathcal{S}}}

\newcommand{\cV}{\ensuremath{\mathcal{V}}}

\newcommand{\s}{\ensuremath{\mathcal{S}}}
\newcommand{\N}{\ensuremath{\mathbb{N}}}

\newcommand{\Vcal}{\ensuremath{\mathcal{V}}}

\newcommand{\Tcal}{\ensuremath{\mathcal{T}}}
\newcommand{\Fcal}{\ensuremath{\mathcal{F}}}
\newcommand{\Xcal}{\ensuremath{\mathcal{X}}}
\newcommand{\Kcal}{\ensuremath{\mathcal{K}}}

\newcommand{\inv}{^{-1}}

\newcommand{\sub}{\subseteq}
\newcommand{\sm}{\smallsetminus}
\newcommand{\es}{\emptyset}

\newcommand{\Aut}{\textnormal{Aut}}

\newcommand{\fwd}{\overset{\scriptscriptstyle\rightarrow}}
\newcommand{\bwd}{\overset{\scriptscriptstyle\leftarrow}}

\DeclareMathOperator{\coloneq}{:\!=}

\theoremstyle{definition}

\theoremstyle{plain}
\newtheorem{conj}{Conjecture}[section]

\newtheorem{lem}[conj]{Lemma}

\newtheorem{thm}[conj]{Theorem}

\newtheorem{prop}[conj]{Proposition}

\theoremstyle{remark}
\newtheorem{remark}[conj]{Remark}

\def\td{tree-decom\-po\-si\-tion}
\def\ta{tree amalgamation}
\def\qt{quasi-tran\-si\-tive}
\def\lf{locally finite}

\newcommand{\comment}[1]{}

\def\?#1{\vadjust{\vbox to 0pt{\vss\vskip-8pt\leftline{%
     \llap{\hbox{\vbox{\pretolerance=-1
     \doublehyphendemerits=0\finalhyphendemerits=0
     \hsize16truemm\tolerance=10000\small
     \lineskip=0pt\lineskiplimit=0pt
     \rightskip=0pt plus16truemm\baselineskip8pt\noindent
     \hskip0pt        %(without this, the first word is never hyphenated!)
     #1\endgraf}\hskip7truemm}}}\vss}}}

\newenvironment{txteq}
{
	\begin{equation}
	\begin{minipage}[t]{0.85\textwidth} % set width to 0.9 x textwidth
	\em                                % switch on emph
}
{\end{minipage}\end{equation}\ignorespacesafterend}
\newenvironment{txteq*}
{
	\begin{equation*}
	\begin{minipage}[t]{0.85\textwidth} % set width to 0.9 x textwidth
	\em                                % switch on emph
}
{\end{minipage}\end{equation*}\ignorespacesafterend}

\begin{document}

\title{Two characterisations of accessible quasi-transitive graphs}
\author{Matthias Hamann\thanks{Supported by the Heisenberg-Programme of the Deutsche Forschungsgemeinschaft (DFG Grant HA 8257/1-1).}\\Mathematics Institute, University of Warwick\\Coventry, UK

\and
Babak Miraftab\\Department of Mathematics, Universit\"at Hamburg\\Hamburg, Germany}

\date{}

\maketitle

\begin{abstract}
We prove two characterisations of accessibility of locally finite \qt\ connected graphs.
First, we prove that any such graph $G$ is accessible if and only if its set of separations of finite order is an $\Aut(G)$-finitely generated semiring.
The second characterisation says that $G$ is accessible if and only if every process of splittings in terms of \ta s stops after finitely many steps.
\end{abstract}

\section{Introduction}

Tree amalgamations are a graph product that offers a way to construct graphs that are, in general, multi-ended.
(We refer to Section~\ref{sec_ta} for its definition.)
On the other hand, every suitable multi-ended graph can be written as a non-trivial \ta, see Theorem~\ref{thm_StallingsGraph}.
Note that \ta s are a graph theoretic analogue of the following two group products: free products with amalgamation and HNN-extensions.
Also, Theorem~\ref{thm_StallingsGraph} is a graph theoretic version of Stallings' splitting theorem of finitely generated groups~\cite{Stallings}.

\begin{thm}\label{thm_StallingsGraph}{\rm \cite[Theorem 5.5]{HLMR}}
Every multi-ended \qt\ locally finite connected graph is a non-trivial \ta\ of two \qt\ locally finite connected graphs of finite adhesion and finite identification, distinguishing ends and respecting the action of the involved groups.
\end{thm}

When $G$ is a \ta\ of $G_1$ and $G_2$ with all properties as in Theorem~\ref{thm_StallingsGraph}, then we say that $(G_1,G_2)$ is a \emph{factorisation} of~$G$ and $G$ \emph{splits} into $G_1$ and~$G_2$.
More generally, a tuple $(G_1,\ldots,G_i)$ of \qt\ locally finite connected graphs is a \emph{factorisation} of~$G$ if $G$ is obtained by iterated non-trivial \ta s of finite adhesion, finite identification and respecting the actions of the involved groups of all these graphs~$G_i$.
A factorisation is \emph{terminal} if all its graphs have at most one end.
We call a graph \emph{accessible} if it has a terminal factorisation.

The question arises which \qt\ locally finite connected graphs are accessible.
A result of~\cite{HLMR} says that such graphs are accessible if and only if they are accessible in the sense of Thomassen and Woess~\cite{ThomassenWoess}.
(We refer to Section~\ref{sec_ta} for their definition of accessibility.)
By examples of Dunwoody~\cite{D-AnInaccessibleGroup,D-AnInaccessibleGraph}, it is known that there are inaccessible \qt\ locally finite connected graphs.
We are looking for characterisation results for accessibility and in this paper we are going to prove two such results.

The first result deals with the set $\s(G)$ of all separations of finite order of \qt\ locally finite connected graphs~$G$.
(For the definition of separations and related notions, we refer to Section~\ref{sec_Process}.)
This set equipped with two natural operations is a semiring and the automorphisms of~$G$ induce an action on~$\s(G)$.
We prove that $G$ is accessible if and only if there are finitely many separations in $\s(G)$ that generate together with their $\Aut(G)$-images the whole semiring $\s(G)$.
We then say that $\s(G)$ is \emph{$\Aut(G)$-finitely generated}.
This characterisation can be considered as a result analogous to~\cite[Corollary IV.7.6]{dicksdun} for the set of separations instead of the cut space.

Before we explain the second characterisation, let us look at factorisations once more.
If $(G_1,G_2)$ is a factorisation of~$G$, we may ask if one of these factors has again more than one end.
If so, we can apply Theorem~\ref{thm_StallingsGraph} to that factor (and the stabiliser of that factor in $\Aut(G)$ as group acting \qt ly on it) and obtain a factorisation of it.
We can repeat this \emph{process of splittings} as long as there are factors with more than one end.
It is clear from the definition that some process of splittings stops if and only if the graph has a terminal factorisation and thus is accessible.
It was conjectured in \cite[Conjecture 6.5]{HLMR} that the property of stopping of the process of splittings does not depend on the particular splittings.
To be precise, it was conjectured that one process of splittings stops after finitely many steps if and only if every process does this.
Our second characterisation confirms this in a strong sense.

\begin{thm}\label{mainthmShort}
Let $G$ be a \qt\ locally finite connected graph. 
Then the following statements are equivalent:
\begin{enumerate}[\rm (i)]
\item\label{itm_mainShort_Acc} $G$ is accessible.
\item\label{itm_mainShort_Semiring} $\cS(G)$ is an $\Aut(G)$-finitely generated semiring.
\item\label{itm_mainShort_Process} Every process of splittings of $G$ must end after finitely many steps.
\item\label{itm_mainShort_Process2} There exists $\kappa(G)\in\N$ such that every process of splittings of~$G$ stops after $\kappa(G)$ steps.
\end{enumerate}	
\end{thm}

Our paper is structured as follows.
In Section~\ref{sec_ta} we are going to define \ta s and all related notions.
In Section~\ref{sec_semiring}, we investigate the semiring $\s(G)$ and prove a major step for the equivalence of~(\ref{itm_mainShort_Acc}) and~(\ref{itm_mainShort_Semiring}) of Theorem~\ref{mainthmShort}.
In Section~\ref{sec_Process}, we are going to prove that (\ref{itm_mainShort_Acc}) implies (\ref{itm_mainShort_Process2}) and hence (\ref{itm_mainShort_Process}) of Theorem~\ref{mainthmShort}.
In Section~\ref{sec_pf}, we fill in the remaining gaps of the proof of Theorem~\ref{mainthmShort}.

\section{Tree amalgamations}\label{sec_ta}

In this section, we will state all notations and results that are needed in the context of \ta s.

A tree $T$ with the canonical bipartition $\{V_1,V_2\}$ of its vertex set
is called \emph{$(p_1,p_2)$-semiregular} if all vertices in~$V_i$ have degree~$p_i$ for $i=1,2$.

Let $e=xy$ be an edge of a graph~$G$ and let $v_e$ be a new vertex.
Let $G'$ be the graph with vertex set $(V(G)\sm\{x,y\})\cup\{v_e\}$ and edges between $u,v\in V(G')\sm\{v_e\}$ if and only if $uv\in E(G)$ and between $u\in V(G')\sm\{v_e\}$ and~$v_e$ if and only if $u$ is adjacent to either $x$ or~$y$ in~$G$.
Then $G'$ is the graph obtained by \emph{contracting} the edge~$e$.
If $E$ is a subset of $E(G)$, then we denote by $G/E$ the graph obtained by contracting all edges in~$E$.

Let~$G_i$ be a graph for $i=1,2$.
Let $I_1$ and $I_2$ be disjoint sets.
Let every $V(G_i)$ have a family $(S_k^i)_{k \in I_i}$ of subsets such that all these subsets have the same cardinality.
For all $k\in I_1$ and $\ell\in I_2$, let $\phi_{k\ell}$ be a bijective map from $S_k^1$ to~$S_\ell^2$.
We set $\phi_{\ell k}\coloneq\phi_{k\ell}\inv$ and call the maps $\phi_{k\ell}$ and $\phi_{\ell k}$ \emph{bonding maps}.

Let~$T$ be a $(|I_1|,|I_2|)$-semiregular tree with the canonical bipartition $\{V_1,V_2\}$ such that the vertices in~$V_i$ have degree~$|I_i|$.
Let $D(T)$ be the set of oriented edges of $E(T)$, i.\,e.\ $D(T)=\{\fwd{uv}\mid uv\in E(T)\}$.
If $\fwd{e}=\fwd{uv}\in D(T)$, then we denote by $\bwd{e} \coloneq \bwd{vu}$ its reverse.
Let $v\in V_i$ and let $E_v$ be the set of all edges in~$D(T)$ starting at~$v$.
Let $f\colon D(T) \to I_1 \cup I_2$ be a labelling such that its restriction to~$E_v$ is a bijection to~$I_i$.

For every $i\in\{1,2\}$ and every $v\in V_i$, let $G_v$ be a copy of~$G_i$. 
Denote by~${S^v_k}$ the corresponding copies of~$S_k^i$ in~$V(G_v)$. 
Let $G_1+G_2$ be the graph obtained from the disjoint union of the graphs $G_v$ for all $v\in V(T)$ by adding new edges between each $x\in S^{u}_k$ and $\phi_{k\ell}(x)\in S^v_\ell$ for every edge $\fwd{e} = \fwd{uv}$ with $f(\fwd{e}) = k$ and $f(\bwd{e}) = \ell$.
The new edges do not depend on the orientation of~$\fwd{e}$ because of $\phi_{\ell k} = \phi_{k\ell}^{-1}$.
Let $F$ be the set of these new edges of $G_1+G_2$.
The \emph{\ta} $G_1\ast_T G_2$, or just $G_1\ast G_2$, of the graphs~$G_1$ and~$G_2$ over the \emph{connecting tree}~$T$ is defined as $(G_1+G_2)/F$.
Let $\pi\colon V(G_1+G_2)\to V(G_1\ast G_2)$ be the canonical map that maps each $x\in V(G_1+G_2)$ to the vertex obtained from~$x$ after all the contractions.

The sets $S_k^i$ and their canonical images in $G_1\ast G_2$ are the \emph{adhesion sets} of the \ta.
The \ta\ has \emph{finite adhesion} if one (and hence all) of its adhesion sets are finite. 
We call a \ta\ $G_1\ast_T G_2$ \emph{trivial} if for some $v\in V(T)$ the restriction of~$\pi$ to~$G_v$ is a bijection.
Note that if the \ta\ has finite adhesion, then it is trivial if, for some $i\in\{1,2\}$, the set $V(G_i)$ is the only adhesion set of~$G_i$ and $|I_i|=1$.

For a vertex $x\in V(G_1\ast G_2)$ let $T_x$ be the maximal subtree of~$T$ such that every node of~$T_x$ contains a vertex $y$ with $\pi(y)=x$.
The \emph{identification size} of a vertex $x\in V(G_1\ast G_2)$ is the cardinality of~$V(T_x)$.
The \ta\ has \emph{finite identification} if the identification sizes of its vertices are bounded.

So far, the \ta\ do not interact with any group action.
In the following, we define some notions that ensure that \ta s of \qt\ graphs that satisfy this notion are again \qt, see \cite[Lemma~5.3]{HLMR}.

For $i=1,2$, let $\Gamma_i$ be a group that acts on~$G_i$.
Let $\{i,j\}=\{1,2\}$.
The \ta\ \emph{respects $\gamma\in\Gamma_i$} if there is a permutation $\pi$ of~$I_i$ such that for every $k \in I_i$ there exists $\ell \in I_j$ and $\tau$ in the setwise stabiliser of~$S_\ell$ in~$\Gamma_j$ such that 
\[
\phi_{k\ell} = \tau \circ \phi_{\pi(k)\ell} \circ \gamma\mid_{S_k}.
\]
The \ta\ \emph{respects $\Gamma_i$} if it respects every $\gamma \in \Gamma_i$.

Let $k\in I_i$ and let $\ell,\ell'\in I_j$. 
The bonding maps from $k$ to $\ell$ and $\ell'$ are \emph{consistent} if there exists $\gamma \in \Gamma_j$ such that 
\[
    \phi_{k\ell} = \gamma \circ \phi_{k\ell'}.
\]
The bonding maps between $J_i \subseteq I_i$ and $J_j\subseteq I_j$ are \emph{consistent} if they are consistent for all $k \in J_i$ and $\ell,\ell' \in J_j$.

The \ta\ $G_1\ast G_2$ is of \emph{Type 1 respecting the actions of~$\Gamma_1$ and $\Gamma_2$} if the following holds:
\begin{enumerate}[(i)]
\item \label{itm:T1respaction} the \ta\ respects $\Gamma_1$ and $\Gamma_2$;
\item \label{itm:T1esim} the bonding maps between $I_1$ and $I_2$ are consistent.
\end{enumerate}

The \ta\ $G_1\ast G_2$ is of \emph{Type 2 respecting the actions of~$\Gamma_1$ and $\Gamma_2$} if the following holds:
\begin{enumerate}[(i)]
\item[(o)] \label{itm:T2prelim} $G_1=G_2=:G$, $\Gamma_1=\Gamma_2=:\Gamma$,  $I_1=I_2=:I$,\footnote{Formally we asked $I_1$ and $I_2$ to be disjoint. We can guarantee this easily by using appropriate bijections.} and there exists $J \subseteq I$ such that $f(\fwd e) \in J$ if and only if $f(\bwd e) \notin J$;
\item \label{itm:T2respaction} the \ta\ respects $\Gamma$;
\item \label{itm:T2esim} the bonding maps between $J$ and $I\sm J$ are consistent.
\end{enumerate}

The \ta\ $G_1\ast G_2$ \emph{respects the actions (of $\Gamma_1$ and $\Gamma_2$)} if it is of either Type~1 or Type~2 respecting the actions $\Gamma_1$ and $\Gamma_2$.

\medskip

A \emph{ray} is a one-way infinite path.
Two rays are \emph{equivalent} if for every finite $S\sub V(G)$ both rays have all but finitely many vertices in the same component of $G-S$.
This is an equivalence relation whose equivalence classes are the \emph{ends} of~$G$
An end is \emph{thick} is is contains infinitely many pairwise disjoint rays.
A \emph{double ray} is a two-way infinite path.

The \ta\ $G=G_1\ast G_2$ \emph{distinguishes} ends if there is some adhesion set $S_k^u=S_\ell^v$ for adjacent $u,v\in V(T)$ such that for every component $C$ of $T-uv$ the graph induced by $\bigcup_{w\in C}G_i^w$ contains an end.

A graph $G$ is \emph{accessible in the sense of Thomassen and Woess} if there is an $n\in\N$ such that any two distinct ends of~$G$ are separable by at most~$n$ edges, that is, there are $n$ edges such that every double ray between these two ends contains one of those $n$ edges.

Let $\Gamma$ be a group acting on a tree~$T$.
The action is \emph{inversion-free} if there is no edge $uv$ of~$T$ and no $\gamma\in\Gamma$ such that $\gamma(u)=v$ and $\gamma(v)=u$.
We then also say that $\Gamma$ acts on~$T$ \emph{without inversion}.

A \emph{\td} of a graph~$G$ is a pair $(T,\cV)$ consisting of a tree~$T$ and a set of vertex sets $V_t$ of~$G$, one for every node of~$T$, such that the following hold:
\begin{enumerate}[(T1)]
\item $V(G)=\bigcup_{t\in V(T)}V_t$;
\item for every $e\in E(G)$ there exists $t\in V(T)$ with $e\sub V_t$;
\item $V_{t_1}\cap V_{t_3}\sub V_{t_2}$ for all $t_1,t_2,t_3\in V(T)$ such that $t_2$ separates $t_1$ and~$t_3$.
\end{enumerate}

The elements of~$\cV$ are the \emph{parts} of the \td.
The sets $V_t\cap V_{t'}$ for edges $tt'\in E(T)$ are the \emph{adhesion sets} of $(T,\cV)$.
If all adhesion sets are finite we say that $(T,\cV)$ has \emph{finite adhesion}.

A \emph{separation} of~$G$ is an ordered pair~$(A,B)$ such that~$G[A]\cup G[B]=G$, that is, $V(G)=A\cup B$ and there is no edge with one end vertex in~$A\sm B$ and the other in $B\sm A$.
The \emph{order} of~$(A,B)$ is~$|A\cap B|$.
A separation $(A,B)$ is \emph{tight} if there are components $C_A$ of $A\sm B$ and $C_B$ of $B\sm A$ such that every $x\in A\cap B$ has neighbours in~$C_A$ and in~$C_B$.

For two separation $(A,B)$ and $(C,D)$, we write $(A,B)\leq (C,D)$ if $A\subseteq C$ and $B\supseteq D$.
We call $(A,B)$ and $(C,D)$ \emph{nested} if they are comparable with respect to this relation and a set of separations is \emph{nested} if every two of its elements are nested.

If $(T,\cV)$ is a \td\ of~$G$, then the separations \emph{induced by} $(T,\cV)$ are those of the form $(\bigcup_{t\in T_1}V_t,\bigcup_{t\in T_2}V_t)$ for edges $t_1t_2\in E(T)$, where $T_i$ is the component of $T-t_1t_2$ that contains~$t_i$.
It follows from (T3) that these are indeed separations and its separator is $V_{t_1}\cap V_{t_2}$.
Furthermore, the set of all separations induced by $(T,\cV)$ is nested.

If a group $\Gamma$ acts on~$G$, a \td\ $(T,\cV)$ is \emph{$\Gamma$-invariant} if the induced action of~$\Gamma$ on~$\cV$ induces an action on~$T$.

Let $(T,\cV)$ and $(T',\cV')$ be \td s of~$G$.
We call $(T',\cV')$ a \emph{refinement} of $(T,\cV)$ if there is a family of disjoint subtrees $(T_i)_{i\in I}$ of~$T'$ covering $V(T')$ such that the following holds:
\begin{enumerate}[(R1)]
\item $T=T'/\bigcup_{i\in I} E(T_i)$;
\item $\bigcup_{s\in T_i} V_s'=V_t$, where $t$ is the node of~$T$ obtained from the contraction of~$E(T_i)$.
\end{enumerate}

Tree amalgamations induce in a canonical way \td s.
This is one of the main properties that we will use in the proofs of our main result.
The way how \ta s $G\coloneq G_1\ast_TG_2$ induce \td s was discussed in~\cite[Remark 5.1]{HLMR}: the pair
\[
(T,\{\pi(V(G_u))\mid u\in V(T)\})
\]
is a \td\ \emph{corresponding to} the factorisation $(G_1,G_2)$ of~$G$ all of whose parts induce connected graphs.

\section{The semiring $\s(G)$}\label{sec_semiring}

A \emph{semiring} is a triple $(R,+,\times)$ such that $(R,+)$ is an abelian monoid, $(R,\times)$ is a monoid and $\times$ is distributive over~$+$.
A semiring $(R,+,\times)$ is \emph{commutative} if $(R,\times)$ is commutative.
A set $S\sub R$ \emph{generates} $R$ if every $r\in R$ is obtained by finitely many additions and multiplications of elements of~$S$.

An immediate corollary of a result by Thomassen and Woess \cite[Proposition 4.2]{ThomassenWoess} is the following.

\begin{prop}\label{prop_corTW}
Let $G$ be a locally finite graph, let $v\in V(G)$ and let $k\in\N$.
Then there are only finitely many tight separations of order $k$ with $v$ in their separator.\qed
\end{prop}

Let $\s(G)$ be the set of all separations of finite order of~$G$.

We define for $(A,B),(C,D)\in\s(G)$ the following operations:
\[
(A,B)+(C,D)\coloneq (A\cap C, B\cup D),
\]
\[
(A,B)\times(C,D)\coloneq (A\cup C, B\cap D).
\]

Simple calculations show that $(\s(G),+,\times)$ is a commutative semiring, where $(V(G),\es)$ is the neutral element with respect to~$+$ and $(\es,V(G))$ is the neutral element with respect to~$\times$.

Let $\s_n(G)$ be the subsemiring of~$\s(G)$ that is generated by the tight separations of order at most~$n$.

\begin{prop}\label{prop_sepGenByTight}
Let $G$ be a locally finite graph.
Every separation of order~$n$ is generated by tight separations of order at most~$n$.
\end{prop}

\begin{proof}
We prove the assertion by induction on the order of the separation.
Let $(A,B)$ be a separation of order~$n$ that is not tight.
Since $(A,B)$ is not tight, either $A\sm B$ or $B\sm A$ has no component $C$ with $A\cap B\sub N(C)$.

If $A\sm B$ has no such component, let $\Kcal_A$ be the set of all components of $A\sm B$.
Then
\[
(A,B)=(X,V(G))\times \prod_{C\in \Kcal_A}(C\cup N(C),V(G)\sm C),
\]
where $X=(A\cap B)\sm \bigcup_{C\in \Kcal_A}N(C)$.
Every separation of that product has order less than $n$.
By induction, $(A,B)$ can be generated by tight separations of order less than~$n$.

If $B\sm A$ has no component $C$ with $A\cap B\sub C$, let $\Kcal_B$ be the set of components of $B\sm A$.
Then
\[
(A,B)=(V(G),Y)+\sum_{C\in\Kcal_B}(V(G)\sm C,C\cup N(C)),
\]
where $Y=(A\cap B)\sm\bigcup_{C\in\Kcal_B}N(C)$.
Every summand of that sum has order less than~$n$, so by induction, $(A,B)$ is generated by tight separations of order at most~$n$.
\end{proof}

A \td\ $(T,\Vcal)$ \emph{distinguishes} two ends \emph{(efficiently)} if the separator of some separation induced by $(T,\Vcal)$ separates those ends \emph{(minimally)}.
We need the following result for the proof of Proposition~\ref{prop_TWAccToS=Sn}.

\begin{thm}{\rm \cite[Theorem 6.4]{HLMR}}\label{thm_HLMR6.3}
Let $G$ be a connected locally finite graph that is accessible in the sense of Thomassen and Woess and let $\Gamma$ be a group acting \qt ly on~$G$.
Then there exists a $\Gamma$-invariant \td\ $(T,\Vcal)$ of~$G$ of finite adhesion such that $(T,\Vcal)$ distinguishes all ends of~$G$ efficiently and such that there are only finitely many $\Gamma$-orbits on $E(T)$.\qed
\end{thm}

The proof of the following result is based on the idea of a proof of Thomassen and Woess~\cite[Theorem 7.6]{ThomassenWoess}.

\begin{prop}\label{prop_TWAccToS=Sn}
Let $G$ be a \qt\ locally finite connected graph that is accessible in the sense of Thomassen and Woess.
Then there exists $n\in\N$ such that $\s(G)=\s_n(G)$.
\end{prop}

\begin{proof}
Let $(T,\Vcal)$ be a \td\ as in Theorem~\ref{thm_HLMR6.3}.
In particular, there are only finitely many $\Aut(G)$-orbits on~$E(T)$.
Let $\Tcal$ be the set of separations that are induced by $(T,\Vcal)$.
Then there are only finitely many $\Aut(G)$-orbits on $\Tcal$ as well.
Let $n_1$ be the maximum order of separations in~$\Tcal$ and let $n_2$ be the maximum degree of~$G$.
Set $n\coloneq\max\{n_1,n_2\}$.
We will show $\s(G)=\s_n(G)$.

Let $(A,B)\in\s(G)$.
Our aim is to show that $(A,B)$ is generated by elements of $\s_n(G)$.
Let $\Omega_A,\Omega_B$ be the set of ends of~$G$ that live in~$A$, in~$B$, respectively.
We claim the following.
\begin{txteq}\label{clm_Ffinite}
There is a finite $\Fcal\sub\Tcal$ such that for every $\omega_A\in\Omega_A$ and every $\omega_B\in\Omega_B$ there exists $(C,D)\in\Fcal$ such that $\omega_A$ lives in~$C$ and $\omega_B$ lives in~$D$.
\end{txteq}
Suppose (\ref{clm_Ffinite}) does not hold.
Let $\{(A_i,B_i)\mid i\in\N\}=\Tcal$.
For every $i\in\N$ let $\omega_i^A\in\Omega_A$ and $\omega_i^B\in\Omega_B$ such that $\omega_i^A$ and $\omega_i^B$ are not distinguished by any $(A_j,B_j)$ with $j\leq i$.
These ends exist as (\ref{clm_Ffinite}) does not hold.
For every $i\in\N$ let $P_i$ be a double ray between $\omega_i^A$ and~$\omega_i^B$ none of whose vertices are separated by any $(A_j,B_j)$ with $j\leq i$.
Every double ray $P_i$ meets the finite vertex set $A\cap B$.
Thus the sequence $(P_i)_{i\in\N}$ has a subsequence that converges to a double ray~$P$: infinitely many $P_i$ share an edge incident with some vertex of $A\cap B$, among which we find an infinite subsequence whose edges adjacent to the first one coincide on each side and so on.
By construction, one tail of~$P$ lies in~$A$ and another tail lies in~$B$.
In particular it has tails in distinct ends of~$G$.
By the choice of the double rays $P_i$, no $(A_i,B_i)$ separates tails of~$P$, that is, the two ends of~$G$ that contain tails of~$P$ are not distinguished by any $(A_i,B_i)$ and thus are not distinguished by $(T,\Vcal)$.
This contradiction to the choice of $(T,\Vcal)$ shows (\ref{clm_Ffinite}).

Let $F$ be the set of edges of~$T$ that corresponds to the finite set $\Fcal$.
Let $V_A$ be the set of nodes of~$T$ that lie in components $C$ of $T-F$ such that some end of~$\Omega_A$ lives in $\bigcup_{t\in C}V_t$.
Set $V_B\coloneq V(T)\sm V_A$.
We consider the separation
\[
(C,D)\coloneq\left(\bigcup_{t\in V_A}V_t,\bigcup_{t\in V_B}V_t\right).
\]
Then $(C,D)$ is generated by~$\Fcal$.
By construction, $\Omega_A$ is the set of ends of~$G$ that live in~$C$ and $\Omega_B$ is the set of ends of~$G$ that live in~$D$.
We shall prove the following.
\begin{txteq}\label{clm_main2}
The sets $A\sm C$ and $C\sm A$ are finite.
\end{txteq}
For every vertex $x\in A\sm C$ and every neighbour $y$ of~$x$ outside of $A\sm C$, we have either $y\in C\cap D$ or $x\in A\cap B$ and $y\in N(A)$.
Since $X\coloneq N(A)\cup (C\cap D)$ is finite, since $G$ is locally finite and since each component of $A\sm C$ is a component of $G-X$, the vertex set $A\sm C$ induces only finitely many components in~$G$.
If one of these components is infinite, there would be an end living in $A\sm C$ which is impossible as we already saw that this set is empty.
Thus $A\sm C$ is finite.
An analogous argument shows that $C\sm A$ is finite.
This completes the proof of~(\ref{clm_main2}).

Since (\ref{clm_main2}) holds, $(A,B)$ and $(C,D)$ differ only by addition and multiplication of elementary separations.
So $(A,B)$ is generated by separations of order at most~$n$.
Proposition~\ref{prop_sepGenByTight} implies that $(A,B)\in\s_n(G)$.
\end{proof}

\section{Iterated splittings}\label{sec_Process}

In this section, we will prove that one process of splittings of a \qt\ \lf\ connected graph stops if and only if every process of splittings of that graph does that.
In preparation for that, we define the following property for all $i\in\N$ and for a factorisation $(G_1,G_2)$ of a \qt\ locally finite connected graph~$G$.
\begin{equation}\tag{$*^{(i)}$}\label{itm_facEnds2}
	\begin{minipage}[t]{0.85\textwidth}	\em Every process of splittings of~$G$ that starts with $(G_1,G_2)$ ends after at most $i$ steps.
\end{minipage}
\end{equation}

\begin{lem}\label{lem_PartContainsS}
Let $(G_1,G_2)$ be a factorisation of a \qt\ locally finite connected graph~$G$ and let $S\sub V(G)$ be a finite set.
Then there exists a factorisation $(H_1,H_2)$ of~$G$ such that
\begin{enumerate}[\rm (i)]
\item\label{itm_PartContainsS_1} $S$ is contained in some part of the \td\ corresponding to $(H_1,H_2)$ and
\item\label{itm_PartContainsS_2} if $(H_1,H_2)$ satisfies (\ref{itm_facEnds2}), then $(G_1,G_2)$ satisfies (\ref{itm_facEnds2}).
\end{enumerate}
\end{lem}

\begin{proof}
Let $(T,\cV)$ be the \td\ corresponding to $(G_1,G_2)$.
Let $S'\sub V(G)$ be finite and connected and such that $S\sub S'$.
Let $T_{S'}$ be the minimal subtree of~$T$ such that for every $t\in V(T-T_{S'})$ we have $V_t\cap S'=\es$.
This subtree is finite since $G_1\ast G_2$ is of finite identification.
For every $t\in V(T)$, we set
\[
V_t':= V_t\cup\bigcup\{\alpha(S')\mid \alpha \in \Aut(G), t\in\alpha(T_{S'})\}
\]
and
\[
\cV':=\{V_t'\mid t\in V(T)\}.
\]
To see that $(T,\cV')$ is a \td\ it suffices to prove~(T3).
For this, it suffices to see that, for every $v\in V(G)$, the subgraph of~$T$ that contains~$v$ is a tree.
But this follows immediately from the definition of the subtrees $T_{S'}$ and the parts~$V_{t'}$.
By construction, $(T,\cV')$ corresponds to a factorisation $(H_1,H_2)$ of~$G$ such that for some edge $t_1t_2\in E(T)$ and every $i\in\{1,2\}$ the graph $H_i$ is isomorphic to the subgraph of~$G$ induced by $V_{t_i}'$.
This proves~(\ref{itm_PartContainsS_1}).

In order to prove (\ref{itm_PartContainsS_2}), we note that there is a canonical bijection between the ends of $G[V_t]$ and those of $G[V_t']$.
Starting with a process of splittings that starts with $(G_1,G_2)$, we obtain one that starts with $(H_1,H_2)$ if we add all vertices of $\alpha(S')$ to each adhesion set if it contains one vertex of $\alpha(S')$.
That way, we obtain a process of splittings that starts with $(H_1,H_2)$.
Thus, if all of those processes that start with $(H_1,H_2)$ stop after at most~$i$ steps, so must the processes starting with $(G_1,G_2)$.
This proves~(\ref{itm_PartContainsS_2}).
\end{proof}

Let us now define recursively, what it means for a \td\ to correspond to a factorisation of more than two factors.
A \td\ $(T,\cV)$ of a graph~$G$ \emph{corresponds to} a factorisation $(G_1,\ldots, G_n)$ of~$G$ if the following hold:
\begin{enumerate}[(i)]
\item there is a factorisation $(H_1,\ldots,H_{n-1})$ of~$G$ and a factorisation $(G_i,G_j)$ of some $H_m$ such that
\[
\{G_k\mid 1\leq k\leq n, k\neq i, k\neq j\}=\{H_\ell\mid 1\leq\ell\leq n-1,\ell\neq m\},
\]
\item there is a \td\ $(T',\cV')$ corresponding to the factorisation $(H_1,\ldots,H_{n-1})$ and a \td\ $(T'',\cV'')$ corresponding to the factorisation $(G_i,G_j)$ such that $(T,\cV)$ is a refinement of $(T',\cV')$ where the only non-trivial subtrees in the covering of~$T$ are those that get contracted to nodes whose parts correspond to~$H_m$ and the \td s induced by those trees are isomorphic to $(T'',\cV'')$ in a canonical way.
\end{enumerate}

While in general there need not be a \td\ corresponding to a given factorisation, we will show in Proposition~\ref{prop_tdIsIncomp} how to alter the factorisation slightly to find a \td\ corresponding to that new factorisation.
In order to prove that result, we need the following lemma.

\begin{lem}\label{lem_NiceSplitting}
Let $G$ and $H$ be \qt\ locally finite connected graphs such that $G=H\ast_T H$ is a \ta\ of Type 1 respecting the group actions such that the induced action of $\Aut(G)$ on the connecting tree $T$ is with inversion of the edges.
Then there exists a finite connected graph $K\not\cong H$ such that $G=H\ast K$.

Furthermore, $(H,H)$ satisfies (\ref{itm_facEnds2}) if and only if $(H,K)$ satisfies (\ref{itm_facEnds2}).
\end{lem}

\begin{proof}
Let $S$ be an adhesion set of the \ta\ in some $G_u$ with $u\in V(T)$ and let $S'\sub V(G_u)$ be connected and finite with $S\sub S'$.
Let $\varphi\in\Aut(G)$ such that it reverses the edge $uv\in E(T)$ whose adhesion set is~$S$.
Then we have $S=\varphi(S)$.
If $S'\neq V(G_u)$, let $K$ be the subgraph of~$G$ induced by $\pi(S')\cup \varphi(\pi(S'))$.
If $S'=V(G_u)$, let $K$ be the subgraph of~$G$ induced by $\pi(G_u)$ and $\pi(G_v)$.
In both cases, the graph $K$ is finite and connected and $H$ and $K$ are not isomorphic.
The \ta\ $H\ast K$ is~$G$ with adhesion sets the copies of $S'$ in the first case and of $V(H)$ in the second case.

The additional assertion holds since each of the two factorisations $(H,H)$ and $(H,K)$ satisfies (\ref{itm_facEnds2}) if and only if every process of splittings of~$H$ stops after at most $i-1$ steps.
\end{proof}

\begin{prop}\label{prop_tdIsIncomp}
Let $(G_1,\ldots,G_k)$ be a factorisation of a \qt\ locally finite connected graph~$G$ and let $(T,\cV)$ be an $\Aut(G)$-invariant inversion-free \td\ corresponding to $(G_1,\ldots,G_k)$ with $|\Aut(G)\setminus E(T)|<\infty$.
Assume that $G_i$ has more than one end and let $(H_1,H_2)$ be a factorisation of~$G_i$.
Then there is a factorisation $(H_1',H_2')$ of $G_i$ such that $(H_1,H_2)$ satisfies (\ref{itm_facEnds2}) if $(H'_1,H'_2)$ satisfies (\ref{itm_facEnds2}) and there is an $\Aut(G)$-invariant inversion-free \td\ $(T',\cV')$ with $|\Aut(G)\setminus E(T')|<\infty$ such that $(T',\cV')$ is a refinement of $(T,\cV)$ that corresponds to $(G_1,\ldots,G_{i-1},G_{i+1},\ldots,G_k,H_1',H_2')$.

Furthermore, we may assume that the \td\ of~$G_i$ corresponding to $(H_1,H_2)$ distinguishes the same ends of~$G$ as the one corresponding to $(H'_1,H'_2)$.
\end{prop}

\begin{proof}
Let $t\in V(T)$ such that $G[V_t]$, the graph induced by~$V_t$, corresponds to the factor~$G_i$.
By iterated applications of Lemma~\ref{lem_PartContainsS} and since there are only finitely many $\Aut(G)$-orbits on $E(T)$, we obtain a factorisation $(H_1',H_2')$ of~$G_i$ such that $(H_1,H_2)$ satisfies (\ref{itm_facEnds2}) if $(H'_1,H'_2)$ satisfies (\ref{itm_facEnds2}) and such that for every adhesion set of $(T,\cV)$ in $G[V_t]$ there is some part of the \td\ $(\tilde{T},\tilde{\cV})$ corresponding to $(H_1',H_2')$ that contains its image under the canonical map $G[V_t]\to G_i$.
If the setwise stabiliser $\Gamma$ of~$G_i$ in~$\Aut(G)$ acts on~$\tilde{T}$ without inversion of the edges\footnote{Technically, we would have to use an injective map of the setwise stabiliser of $G[V_t]$ in $\Aut(G)$ to $\Aut(G_i)$.
For the sake of simplicity, we omit this map.}, set $H_2'':= H_2'$.
Otherwise, $H_1'$ is isomorphic to~$H_2'$.
We apply Lemma~\ref{lem_NiceSplitting} to obtain another factorisation $(H_1',H_2'')$ of~$G_k$, where $H_2''$ is finite and $H_1'\not\cong H_2''$ and such that $(H_1',H_2'')$ satisfies (\ref{itm_facEnds2}) if and only if $(H_1', H_2')$ satisfies~(\ref{itm_facEnds2}).
Let $(T^\circ,\cV^\circ)$ be the \td\ that corresponds to $H_1'\ast H_2''$.
Note that $\Gamma$ acts without inversion on~$T^\circ$ and that for every adhesion set of $(T^\circ,\cV^\circ)$ in $G[V_t]$ there is some part of the \td\ $(\tilde{T},\tilde{\cV})$ corresponding to $(H_1',H_2'')$ that contains its image under the canonical map $G[V_t]\to G_i$.

For an adhesion set $S$ of $(T,\cV)$ that lies in~$V_t$, let $T_S$ be the maximal subtree of~$T^\circ$ such that for all $v\in V(T_S)$ the part $V_v^\circ$ contains~$S$.
This is a finite tree since our \ta s are of finite identification and it is non-empty by construction.
Thus, it has either a central vertex or a central edge.
If it is a central vertex, let $v_S$ be this vertex.
If it is a central edge, choose a vertex $v_S$ that is incident with that edge so that for every adhesion set $S'=\alpha(S)$ with $\alpha\in\Aut(G)$ we have $v_{S'}=\alpha(v_S)$.
This is possible since $\Gamma$ acts on~$T^\circ$ without inversion.
Let $(T',\cV')$ be the \td\ that is a refinement of~$(T,\cV)$ where only the trees for vertices in the $\Aut(G)$-orbits of~$t$ are non-trivial and for these we take the \td\ $(T^\circ,\cV^\circ)$ and its $\Aut(G)$-images.

Note that $\Aut(G)$ acts on~$\tilde{T}$ without inversion since this is true for the action of $\Aut(G)$ on~$T$ and of~$\Gamma$ on~$T^\circ$.

For the additional statement, note that the ends contained in a single part of the first \td\ of $G_i$ are precisely the ends contained in some part of the second \td\ since we just enlarged the parts finitely many times by finite vertex sets in the above construction.
\end{proof}

\begin{remark}\label{rem_taType1}
If $(G_1,G_2)$ is a factorisation of a \qt\ locally finite connected graph~$G$ such that the \td\ corresponding to that \ta\ has only one orbit on the vertex set of the tree and the \ta\ is of Type~1, then Proposition~\ref{prop_tdIsIncomp} implies that we find a different factorisation $(G_1',G_2')$ such that the \ta\ is of Type~1 but $\Aut(G)$ does not act transitively on the tree of the corresponding \td.
\end{remark}

Let $(T,\cV)$ and $(T',\cV')$ be two \td s of~$G$ that correspond to two factorisations of~$G$ and assume that the factorisation for $(T,\cV)$ is terminal.
If there is a part $V_t$ of $(T',\cV')$ that contains more than one end, then there exists a separation $(A,B)$ that is induced by $(T,\cV)$ and that distinguishes those ends, since that \td\ corresponds to a terminal factorisation.
This separation $(A,B)$ induces a separation $(A',B'):=(A\cap V_t,B\cap V_t)$ of $G[V_t]$.
Since $(T,\cV)$ is an $\Aut(G)$-invariant \td, $(A',B')$ is nested with all its images under the stabiliser of $V_t$ in $\Aut(G)$.
Thus, the set of these images induces a \td\ $(T_t,\cV_t)$ by \cite[Theorem 3.2]{cantreet.d} and this \td\ distinguishes some ends of~$G[V_t]$, namely precisely those that are distinguished by $(A,B)$ and lie in $G[V_t]$.

We want to introduce two types of refinements, namely one-step and full refinements.
Following Section 4 in~\cite{HLMR}, in particular Proposition 4.1 and Corollary 4.3, we obtain an $\Aut(G)$-invariant \td\ $(T'_t,\cV'_t)$ with a unique $\Aut(G)$-orbit on $E(T'_t)$ and all adhesion sets finite that distinguishes the same ends as $(T_t,\cV_t)$.\footnote{Such a \td\ is called \emph{splitting} in~\cite{HLMR}.}
Then Lemma 5.9 and Theorem 5.10 in~\cite{HLMR} with their proofs imply that there is a factorisation $(G_1,G_2)$ of $G[V_t]$ that corresponds to $(T'_t,\cV'_t)$.
Proposition~\ref{prop_tdIsIncomp} then implies the existence of a factorisation $(G'_1,G'_2)$ of $G[V_t]$ such that the corresponding \td\ is a refinement of $(T',\cV')$ and such that the ends distinguished by that \td\ which lie in $V_t$ are precisely those that are distinguished by $(T'_t,\cV'_t)$.
We call the resulting \td\ a \emph{one-step refinement of $(T',\cV')$ by $(T,\cV)$}.

Now if we recursively do one-step refinements by $(T,\cV)$ starting with $(T',\cV')$ and first do one-step refinements originating from one $\Aut(G)$-orbit of separations induced by $(T,\cV)$, then move on to the next orbit and so on, we must stop after finitely many steps since we only have finitely many $\Aut(G)$-orbits of edges of~$T$ and the separator of each separation induced by one edge only meets finitely many parts and each part only finitely many times.
(Note that the number of times that it meets parts and induces a separation that distinguishes ends of that part may grow in previous steps but not while considering that separation.)
Once the recursion stops, we call the resulting \td\ a \emph{full refinement of $(T',\cV')$ by $(T,\cV)$}.
By our discussion, we directly obtain the following result.

\begin{prop}\label{prop_fullRefine}
Let $(T,\cV)$ and $(T',\cV')$ be two \td s of~$G$ that correspond to two factorisations of an accessible \qt\ locally finite connected graph~$G$ and assume that the factorisation for $(T,\cV)$ is terminal.
Then there exists a full refinement of $(T',\cV')$ by $(T,\cV)$ and a terminal factorisation of~$G$ to which that refinement corresponds.\qed
\end{prop}

Let $\Gamma$ be a group acting on a tree~$T$.
An edge $e = uv \in E(T)$ is \emph{$\Gamma$-compressible} if $\Gamma u \neq \Gamma v$ and either $\Gamma_v = \Gamma_e$ or
$\Gamma_u = \Gamma_e$.

Set $T_0 := T$.
For $i \geq 1$, let $E_i$ be an orbit of $\Gamma$-compressible edges of $T_{i-1}$ and let $T_i$ be the tree obtained from $T_{i-1}$ by contracting $E_i$.
If $\Gamma \setminus E(T)$ is finite, then there is some $i \geq 0$ such that $T_i$ has no $\Gamma$-compressible edge.
Set $\cC(T) := T_i$ and let $c\colon V(T) \to V(\cC(T))$ be the canonical map defined by all contractions, i.\,e.\ a vertex is mapped to the vertex it ends up as after doing all contractions. 
Note that, in general, $\cC(T)$ is not uniquely defined but relies on the choices of the edge sets $E_i$.
If $(T, \cV)$ is a $\Gamma$-invariant \td\ of a graph~$G$, then the pair
\[
\left(\cC(T), \cV^{\cC(T)} := \left\{\bigcup\{V_s \in \cV \mid c(s) = t\} \mid t \in V(\cC(T))\right\}\right)
\]
is also a $\Gamma$-invariant \td\ of~$G$ and we denote it by $\cC(T,\cV)$.
There is a canonical factorisation of~$G$ associated to that \td: the graphs induced by the parts, one for each orbit, form a factorisation of~$G$ whose corresponding \td\ in $\cC(T,\cV)$.

\begin{prop}\label{prop_BoundVertexOrbits}
Let $(T,\cV)$ and $(T',\cV')$ be two \td s obtained from two terminal factorisations of an accessible \qt\ \lf\ connected graph~$G$.
If the \td s have no $\Aut(G)$-compressible edges and if $\Aut(G)$ acts without inversion on $T$ and $T'$, then
\[
|V(T)/\Aut(G)|=|V(T')/\Aut(G)|.
\]
\end{prop}

\begin{proof}
Let $t\in V(T)$.
If $V_t$ is infinite, then it contains a unique thick end.
Thus, there is a unique $t'\in V(T')$ that contains this thick end.
Note that the stabiliser of $t$ in $\Aut(G)$ fixes the thick end and thus also fixes $t'$.
We set $\varphi(t):=t'$.
Let us assume that $V_t$ is finite.
Since the \ta s have finite identification, there exists a minimal finite subtree $T_t$ of $T'$ such that all $t'\in V(T')$ with $V_{t'}\cap V_t\neq\es$ lie in $T_t$.
This finite tree has a central vertex or central edge, which must be fixed by the stabiliser $\Gamma$ of~$t$ in $\Aut(G)$.
Since the action of $\Aut(G)$ on the tree is without inversion, $\Gamma$ fixes both incident vertices of that central edge in the second case.
Thus, $\Gamma$ fixes a vertex $t'$ of~$T_t$ in both cases.
We set $\varphi(t):=t'$ and for $\alpha\in \Aut(G)$, we set $\varphi(\alpha(t)):=\alpha(t')$.
That way, we have defined a map $\varphi\colon V(T)\to V(T')$.
Analogously, we define a map $\psi\colon V(T')\to V(T)$.
Note that vertices in the same $\Aut(G)$-orbit are mapped to vertices in the same $\Aut(G)$-orbit.

Let $t\in V(T)$ and let us assume that $t\neq \psi(\varphi(t))$.
Then there exists a unique non-trivial $t$--$\psi(\varphi(t))$ path~$P$ in~$T$.
By the definition of~$\varphi$ and $\psi$, we obtain that the stabiliser of~$t$ in~$\Aut(G)$ also stabilises $\psi(\varphi(t))$.
Thus, it must stabilise all vertices on~$P$ and since the action on $T$ is without inversion, the stabilisers of~$t$ and of the edge $st$ on~$P$ coincide.
Since $(T,\cV)$ has no $\Aut(G)$-compressible edge, the vertices $s$ and $t$ must lie in the same orbit.
Inductively, all vertices on~$P$ lie in the same orbit as~$t$.
Thus, $\varphi$ and $\psi$ are inverse to each other when lifted to maps between the $\Aut(G)$-orbits.
This proves the assertion.
\end{proof}

The following lemma follows directly from the definition of the types of \ta s.

\begin{lem}\label{lem_OrbitsForTypesTA}
Let $(G_1,\ldots,G_n)$ be a factorisation of~$G$ and let $(H_1,H_2)$ be a factorisation of $G_1$.
Let $(T,\cV)$ and $(T',\cV')$ be \td s corresponding to the factorisations $(G_1,\ldots,G_n)$ and $(H_1,H_2,G_2,\ldots,G_n)$, respectively.
\begin{enumerate}[\rm (i)]
\item If $G_1$ is a \ta\ of Type~$1$ of $H_1$ and~$H_2$, then we have
\[
|V(T')/\Aut(G)|=|V(T)/\Aut(G)|+1
\]
and
\[
|E(T')/\Aut(G)|=|E(T)/\Aut(G)|+1.
\]
\item If $G_1$ is a \ta\ of Type~$2$ of $H_1$ and~$H_2$, so in particular $H_1=H_2$, then we have
\[
|V(T')/\Aut(G)|=|V(T)/\Aut(G)|
\]
and
\[
\pushQED{\qed}
|E(T')/\Aut(G)|=|E(T)/\Aut(G)|+1.\qedhere
\popQED
\]
\end{enumerate}
\end{lem}

A separation $(A,B)$ of~$G$ satisfies ($\ddagger$) if every vertex $x\in A\cap B$ satisfies one of the following statements:
\begin{enumerate}[($\ddagger$1)]
\item $x$ has no neighbours in $A\sm B$ and $B\sm A$;
\item $x$ has neighbours in $A\sm B$ and in $B\sm A$;
\item $x$ has neighbours in either $A\sm B$ or $B\sm A$ but not both and there exists $\varphi\in\Aut(G)$ such that $\varphi((A,B))=(B,A)$ and $\varphi(x)$ is adjacent to $x$.
\end{enumerate}

We say that a separation $(A,B)$ of finite order in a graph $G$ \emph{distinguishes} ends $\omega$ and $\eta$ if the rays in~$\omega$ lie in~$A$ eventually and the rays in~$\eta$ lie in~$B$ eventually or vice versa.

\begin{prop}\label{prop_tighten}
Let $G$ be a \qt\ \lf\ connected graph.
Let $\cN$ be a nested $\Aut(G)$-invariant set of separations of finite order consisting of only finitely many $\Aut(G)$-orbits.
Then there exists a nested $\Aut(G)$-invariant set $\cM$ of separations of finite order satisfying {\rm ($\ddagger$)} and a surjective map $\alpha\colon \cN\to\cM$ such that each element of~$\cN$ distinguishes the same ends as its image in~$\cM$.
\end{prop}

\begin{proof}
Note that every vertex lies in only finitely many separators of separations in~$\cN$ since there are only finitely many orbits in~$\cN$.
Let us assume that there is a separation $(A,B)$ in~$\cN$ that does not satisfy ($\ddagger$).
By exchanging $(A,B)$ by $(B,A)$, if necessary, we may assume that there exists $a\in A\cap B$ with a neighbour in $A\sm B$ but with no neighbour in $B\sm A$ and that there exists no $\varphi\in\Aut(G)$ with $\varphi((A,B))=(B,A)$ and such that $a$ and $\varphi(a)$ are adjacent.
Let us show that we may assume that there is no $(A',B')>(A,B)$ such that $a\in A'\cap B'$ and $a$ adjacent to $A'\sm B'$ but not to $B'\sm A'$.
Indeed, if such a maximal $(A',B')$ exists and we could not replace $(A,B)$ by $(A',B')$, then $a$ must be adjacent to $\varphi(a)$ for some $\varphi\in\Aut(G)$ with $\varphi((A',B'))=(B',A')$, but then $\varphi(a)$ has neighbours in
\[
\varphi(A)\sm \varphi(B)\sub \varphi(A')\sm \varphi(B')= B'\sm A'\sub B\sm A
\]
but not in $\varphi(B)\sm \varphi(A)$.
In particular, we have $a\in\varphi(A)\cap\varphi(B)$.
This contradicts the maximality of $(A',B')$ since 
\[
(A',B')=(\varphi(B'),\varphi(A'))<(\varphi(B),\varphi(A)).
\]

We set
\begin{align*}
\alpha_1((A,B)):=(&A\sm\{\varphi(a)\mid \varphi\in\Aut(G),\varphi((A,B))=(B,A)\},\\
&B\sm\{\varphi(a)\mid \varphi\in\Aut(G),\varphi((A,B))=(A,B)\}),
\end{align*}
\begin{align*}
\alpha_1((B,A)):=(&B\sm\{\varphi(a)\mid \varphi\in\Aut(G),\varphi((A,B))=(A,B)\},\\
&A\sm\{\varphi(a)\mid \varphi\in\Aut(G),\varphi((A,B))=(B,A)\})
\end{align*}
and
\[
\alpha_1(\varphi((A,B))):=\varphi(\alpha_1((A,B))),
\]
\[
\alpha_1(\varphi((B,A))):=\varphi(\alpha_1((B,A)))
\]
for all $\varphi\in\Aut(G)$.
For all elements of $\cN$ that do not lie in a common $\Aut(G)$-orbit with $(A,B)$, we let $\alpha_1$ be the identity on them.
We shall prove the following.
\begin{txteq}\label{eq_MakeAlmostTight}
The set $\alpha_1(\cN)$ is nested and $\Aut(G)$-invariant and, for every $(C,D)\in\cN$, the separations $(C,D)$ and $\alpha_1((C,D))$ distinguish the same ends.
\end{txteq}
Once we have proved~(\ref{eq_MakeAlmostTight}), we apply recursion and obtains new maps $\alpha_2,\ldots, \alpha_k$ for some $k\in\N$ till for no separation $(C,D)$ in the resulting set $\alpha_k\circ\cdots\circ \alpha_1(\cN)$ of separations its separator contains a vertex $x$ that does not satisfy ($\ddagger$1)--($\ddagger$3).
Then all separations in the resulting set $\cM$ of separations satisfy ($\ddagger$).
It follows from~(\ref{eq_MakeAlmostTight}) that all other claimed properties hold for $\cM$ and $\alpha:=\alpha_k\circ \cdots\circ \alpha_1$.
Thus, it suffices to prove~(\ref{eq_MakeAlmostTight}).

By construction, the set $\alpha_1(\cN)$ is $\Aut(G)$-invariant.
Since $a$ is not adjacent to any $\varphi(a)$ with $\varphi\in\Aut(G)$ such that $\varphi((A,B))=(B,A)$, the pair $\alpha_1((C,D))$ is a separation for every $(C,D)\in\cN$.
Obviously, all separations $(C,D)$ with $\alpha_1((C,D))=(C,D)$ distinguish the same ends as their images under~$\alpha_1$.
By symmetry, it suffices to prove that $(A,B)$ and $(A',B'):=\alpha_1((A,B))$ distinguish the same ends.
So let $\omega,\eta$ be two ends that are distinguished by $(A,B)$.
Then one of them has only rays that lie in $A\sm B$, and thus in $A'\sm B'$, eventually, and the other one has only rays that lie in $B\sm A$, and thus in $B'\sm A'$, eventually.
This implies that $\omega$ and $\eta$ are distinguished by $(A',B')$, too.
By a symmetric argument, all ends that are distinguished by $(A',B')$ are distinguished by $(A,B)$, too.
Thus, $(A,B)$ and $(A',B')$ distinguish the same ends.

Let $(C,D),(E,F)\in\cN$ and let $(C',D'):=\alpha_1((C,D))$ and $(E',F'):=\alpha_1((E,F))$.
Assume that $(C,D)<(E,F)$.
Let us first suppose that $F'\not\sub D'$.
Since $F'\sub F\sub D$, there exists a vertex in $D\sm D'$ and hence there exists $\varphi\in\Aut(G)$ such that $C=\varphi(A)$ and $D=\varphi(B)$.
We have
\[
F'\sub D'\cup\{\psi(a)\mid \psi\in\Aut(G),\psi((A,B))=(C,D)\}.
\]
By the maximal choice of $(A,B)$, there exists a neighbour of $\varphi(a)$ in $F\sm E\sub D\sm C$ which contradicts the choice of~$a$.
Thus, we have $F'\sub D'$.

Let us now suppose $C'\not\sub E'$.
Then there exists $\varphi\in\Aut(G)$ with $\varphi(A)=F$, $\varphi(B)=E$ and $\varphi(a)\in C'$.
Since $\varphi(a)$ must have a neighbour in $F\sm E\sub D\sm C$, it must lie in~$D'$.
Thus, we have $\varphi(a)\in C'\cap D'$.
By the maximality of $(A,B)$, either $\varphi(a)$ has a neighbour in $C\sm D\sub E\sm F$ or it has no neighbour in $D\sm C\sub F\sm E$.
Both cases contradict the choice of~$\varphi(a)$.
This contradiction shows $C'\sub E'$ and hence we have $(C',D')\leq(E',F')$.
Thus, $\cN$ is nested, which implies~(\ref{eq_MakeAlmostTight}) and finishes the proof as discussed above.
\end{proof}

\begin{prop}\label{prop_BoundTATyp2}
There exists $N\in\N$ such that, for every \td\ $(T,\cV)$ of a \qt\ \lf\ connected graph~$G$ that corresponds to a factorisation of~$G$, we have
\[
|E(T)/\Aut(G)|\leq|V(T)/\Aut(G)|+N.
\]
\end{prop}

\begin{proof}
By Lemma~\ref{lem_OrbitsForTypesTA}, it suffices to show that there exists $N\in\N$ such that among the \ta s involved in a process of splittings that leads to a factorisation of~$G$ there are at most $N$ \ta s of Type~$2$.

By contracting edges from the \td\ that belong to \ta s of Type~$1$, joining the parts whose tree nodes are incident with the edges and adjusting the factorisation accordingly, we may assume that the \td\ corresponds to a factorisation that uses only \ta s of Type~$2$.

Let $(T,\cV)$ and $(T',\cV')$ be two \td s corresponding to factorisations that are obtained by processes of splittings that use only \ta s of Type~$2$ and such that the latter is a refinement of the first with one additional factorisation of Type~$2$.
It follows from Lemma~\ref{lem_OrbitsForTypesTA} that each of $V(T)$ and $V(T')$ consists of a unique $\Aut(G)$-orbit.

If an adhesion set of $(T,\cV)$ covers a part of $(T',\cV')$, then this part is finite and thus will not be splitted any further.
Since $V(T')$ consists of a single $\Aut(G)$-orbit, this happens only once and then the factorisation corresponding to that \td\ is terminal.

Let us now assume that no adhesion set of $(T,\cV)$ covers any part of $(T',\cV')$.
Note that also no adhesion set of $(T',\cV')$ covers any part of $(T',\cV')$.
We now apply Proposition~\ref{prop_tighten} to the set of separations induced by the \td\ $(T',\cV')$.
Let $\cM$ be the resulting set of separations and let $(T'',\cV'')$ be the \td\ corresponding to~$\cM$, cp.\ \cite[Lemma 2.7]{EKT2022}.
Then we obtain that for every adhesion set $S$ corresponding to an edge $e_S$ in~$T''$ there is an edge $e=xy\in E(G)$ inside a part $V''_t$ such that $t\in e_S$ and such that $y\in S$ and $x\in V''_t\sm S$.
By the choice of~$(T'',\cV'')$, there exists either an edge $f=yz$ or a path $P=yy'z$ of length two such that in both cases $z$ is separated from~$x$ by~$S$ and moreover lies in a part $V''_{t'}$ for $t'\in V(T'')$ such that $e_S$ separates $t$ and $t'$: the latter case follows from the case ($\ddagger$3) in the definition of~($\ddagger$).
In particular those end vertices do not lie in~$S$.
Note that in the first case $xyz$ is an induced path of length~$2$ and in the second case $xyy'z$ is an induced path of length~$3$.
Let us denote the resulting path of length $2$ or~$3$ by~$Q$.
We say that the pair $(x,z)$ is \emph{defined by}~$e_S$.

As there are only finitely many orbits of pairs $(u,v)$ of vertices of distance $2$ or~$3$, it suffices to show that no two adhesion sets define the same pair of vertices.
So let us suppose that there were two distinct adhesion sets defining the same pair $(u,v)$.
Since both separations induced by the edges corresponding to the adhesion sets separate $u$ from~$v$, their nestedness implies that they must lie on a path in~$T''$ between the two subtrees each of whose nodes contains $x$ or~$y$, respectively.
This contradicts the choice of $(u,v)$ as there is a unique edge between those two subtrees that is incident with the subtree whose nodes contain~$u$.

This shows that the number of $\Aut(G)$-orbits of $E(T'')$, and thus of $E(T')$ is bounded by the sum of $1$ and the number of $\Aut(G)$-orbits of pairs of vertices of distance $2$ or~$3$.
\end{proof}

\begin{remark}
Proposition~\ref{prop_BoundTATyp2} implies for inaccessible \qt\ \lf\ connected graphs that all but finitely many splittings involved in any splitting process correspond to \ta s of Type~$1$.
\end{remark}

Now, we are ready to prove the main result of this section.

\begin{thm}\label{thm_iterated}
Let $G$ be a \qt\ \lf\ connected graph. If some process of splittings stops after finitely many steps, then there exists $n\in\N$ such that every process of splittings stops after at most $n$ steps.
\end{thm}

\begin{proof}
Let $(G_1, \ldots , G_n)$ be a terminal factorisation of $G$ that is the result of a process of splittings and let $(T, \cV)$ be the tree-decomposition corresponding to that factorisation.
Let us suppose that there exists an infinite process of splittings.
By Proposition~\ref{prop_tdIsIncomp}, we may assume that the splitting in each step of the infinite process was done so that the factorisation $(G_{i,1}, \ldots , G_{i,n_i})$ of~$G$ obtained in the $i$-th step of that process gives rise to a corresponding \td $(T_i,\cV_i)$ such that $(T_{i+1}, \cV_{i+1})$ is a refinement of $(T_i, \cV_i)$ as in that proposition.
By Proposition~\ref{prop_tdIsIncomp} and Remark~\ref{rem_taType1}, we may assume that $\Aut(G)$ acts on all trees $T$ and $T_i$ without inversion, since it does so on the trivial tree on one vertex, and that we add in each step at least one new $\Aut(G)$-orbit of vertices, if it came from a \ta\ of Type~1.
By construction, $E(T_i)/\Aut(G)$ is finite.
For every $i\in\N$, let $(T_i',\cV_i')$ be a full refinement of $(T_i,\cV_i)$ by $(T,\cV)$ and set $(T_i'',\cV_i''):=\cC(T_i',\cV_i')$.
Let $(G_{i,1}', \ldots , G_{i,n_i}')$ be the terminal factorisation to which $(T_i',\cV_i')$ corresponds by Proposition~\ref{prop_fullRefine} and let $(G_{i,1}'', \ldots , G_{i,n_i}'')$ be the terminal factorisation to which $(T_i'',\cV_i'')$ corresponds.

Let us look at the construction of $\cC(T)$.
We claim the following.
\begin{txteq}\label{eq_iterated_1}
In each step $j$ of the construction of $\cC(T)$ the graph $H$ induced by the set $E_j$ of edges that get contracted is a disjoint union of stars each of which contains at most one node $t$ whose part in the corresponding tree-decomposition is infinite.
\end{txteq}
Let $uv$ be an $\Aut(G)$-compressible edge in some step. We may assume that the stabiliser of $uv$ is the stabiliser of~$u$.
Let $w$ be a neighbour of~$u$ such that $uv$ and $uw$ lie in the same $\Aut(G)$-orbit, i.\,e.\ there is some $\varphi\in \Aut(G)$ with $\varphi(uv) = uw$.
By definition of compressible edges, we have $\varphi(u) = u$, so $\varphi$ lies in $\Aut(G)_u = \Aut(G)_{uv}$.
Thus, $\varphi$ fixes $v$ and hence $u$ has degree $1$ in~$H$.
Since the stabiliser of a leaf acts quasi-transitively on the part of that leaf, an infinite part would imply the existence of a second adhesion set in that part that lies in the same orbit as the first one.
But then the stabiliser of the leaf cannot stabilise the incident edge, which contradicts compressibility.
Thus, all parts of leaves of the star are finite.
So at most the part of one node of the star, its center, is infinite.
This proves (\ref{eq_iterated_1}).

Let us show
\begin{equation}\label{eq_iterated_2}
|V(T_i')/\Aut(G)|\leq |V(T_i'')/\Aut(G)|+\sum_{v\in Y}(d_Y(v)-1),
\end{equation}
where $Y$ is a set of representatives of the $\Aut(G)$-orbits on $V(T_i'')$ and $d_Y(v)$ denotes the number of $\Aut(G)$-orbits of $E(T_i'')$ that contain edges that are incident with~$v$.

We note that due to~(\ref{eq_iterated_1}), the quantity $\sum_{v\in X}(d_X(v)-1)$ strictly increases, where $X$ is a set of representatives of the $\Aut(G)$-orbits on the vertices of the tree of the $j$-th step of the construction of $V(T_i'')$, while the number of $\Aut(G)$-orbits is decreased by exactly one in each step.
This shows~(\ref{eq_iterated_2}).

Thus, we obtain
\begin{align*}
|V(T_i)/\Aut(G)|&\leq |V(T_i')/\Aut(G)|\\
&\leq |V(T_i'')/\Aut(G)|+\sum _{v\in Y}(d_Y(v)-1)\\
&\leq \sum_{v\in Y} d_Y(v)\\
&\leq 2|E(T_i'')/\Aut(G)|,
\end{align*}
where $Y$ is a set of representative of the $\Aut(G)$-orbits in $V(T_i'')$.
Note that the first inequality holds by construction of $(T_i',\cV_i')$.

According to Proposition~\ref{prop_BoundTATyp2}, there exists $N\in\N$ depending only on~$G$ such that
\[
|E(T_i'')/\Aut(G)|\leq|V(T_i'')/\Aut(G)|+N.
\]
Thus, we obtain
\[
|V(T_i)/\Aut(G)|\leq 2|V(T_i'')/\Aut(G)|+2N.
\]

Note that we obtain the same inequality for $(T,\cV)$ and $\cC(T,\cV)$.
Then Proposition~\ref{prop_BoundVertexOrbits} implies
\[
|V(\cC(T))/\Aut(G)|=|V(T_i'')/\Aut(G)|.
\]
Thus, $|V(T_i)/\Aut(G)|$ is bounded in terms of $|V(\cC(T))/\Aut(G)|$, which implies that the number of steps of the splitting process before it stops is bounded in terms of $|V(\cC(T))/\Aut(G)|$.
\end{proof}

\section{The main theorem}\label{sec_pf}

In this section, we are going to prove our main result, Theorem~\ref{mainthm}.
Theorem~\ref{mainthmShort} follows immediately from Theorem~\ref{mainthm}.

\begin{thm}\label{mainthm}
Let $G$ be a \qt\ locally finite connected graph. 
Then the following statements are equivalent.
\begin{enumerate}[\rm (i)]
\item\label{itm_main_Acc} $G$ is accessible.
\item\label{itm_main_TWAcc} $G$ is accessible in the sense of Thomassen and Woess.
\item\label{itm_main_FinGenAlg} $\cS(G)$ is an $\Aut(G)$-finitely generated semiring.
\item\label{itm_main_S=Sn} There is an $n\in \N$ such that $\cS_n(G)=\cS(G)$.
\item\label{itm_main_ProcStop} Every process of splittings of~$G$ must end after finitely many steps.
\item\label{itm_main_Process2} There exists $\kappa(G)\in\N$ such that every process of splittings of~$G$ stops after $\kappa(G)$ steps.
\end{enumerate}	
\end{thm}

\begin{proof}
The equivalence of~(\ref{itm_main_Acc}) and~(\ref{itm_main_TWAcc}) is~\cite[Theorem 6.3]{HLMR}.
The implication (\ref{itm_main_TWAcc}) to~(\ref{itm_main_S=Sn}) holds by Proposition~\ref{prop_TWAccToS=Sn}.
Theorem~\ref{thm_iterated} implies the implication (\ref{itm_main_Acc}) to (\ref{itm_main_Process2}).
The implications (\ref{itm_main_Process2}) to (\ref{itm_main_ProcStop}) and (\ref{itm_main_ProcStop}) to (\ref{itm_main_Acc}) are trivial.

To prove (\ref{itm_main_S=Sn}) to (\ref{itm_main_TWAcc}), let $n\in\N$ such that $\s_n(G)=\s(G)$.
Let $\omega_1$ and $\omega_2$ be ends of~$G$.
We are going to show that there is a separation of order at most~$n$ distinguishing $\omega_1$ and $\omega_2$.
Let $(A,B)$ be a separation of finite order separating $\omega_1$ and $\omega_2$.
Since $\s_n(G)=\s(G)$, there are separations $(A_1,B_1),\ldots,(A_m,B_m)$ of order at most~$n$ such that $(A,B)$ is generated by these separations.
Note that neither the sum nor the product of two separations distinguishes two ends if none of the summands or factors does so.
Thus, there is some $i\in\{1,\ldots,m\}$ such that $(A_i,B_i)$ distinguishes $\omega_1$ and $\omega_2$.
Hence, (\ref{itm_main_TWAcc}) holds.

It remains to prove the equivalence of (\ref{itm_main_FinGenAlg}) and (\ref{itm_main_S=Sn}).
If (\ref{itm_main_FinGenAlg}) holds, let $\Xcal$ be an $\Aut(G)$-invariant generating set that consists of finitely many orbits.
Let $n$ be the maximum order of separations in~$\Xcal$.
Then $\s(G)=\s_n(G)$ by Proposition~\ref{prop_sepGenByTight} and (\ref{itm_main_S=Sn}) holds.

Assume that there is some $n\in\N$ such that $\s(G)=\s_n(G)$.
By Proposition~\ref{prop_corTW} and as $G$ is \qt, there are only finitely many orbits of tight separations of order at most~$n$.
Thus, $\s_n(G)$ and hence $\s(G)$ is $\Aut(G)$-finitely generated.
\end{proof}

\section*{Acknowledgement}

We thank the anonymous referee for pointing out a major gap in an earlier version of this paper.

\end{document}